\documentclass{amsart}
\usepackage{amsmath,amsthm,amssymb,amsfonts}
\usepackage[colorlinks=true]{hyperref}
\usepackage{cleveref}

\hypersetup{urlcolor=blue, citecolor=blue}

\newcommand{\de}{\delta}    
  \newcommand{\ep}{\varepsilon}
   \newcommand{\La}{\Lambda}

\newcommand{\ga}{\gamma}    
\newcommand{\R}{\mathbb{R}}

\newcommand{\Sp}{\mathbb{S}}

\def\H{{\mathbb{H}}}

\newcommand{\pt}{\partial_t}\newcommand{\pa}{\partial}
\newcommand{\les}{{\lesssim}}
\newcommand{\beeq}{\begin{equation}}\newcommand{\eneq}{\end{equation}}

\def\<{\langle}             \def\>{\rangle}

\newenvironment{prf}{\noindent {\bf Proof.} }{\endprf\par}
\def \endprf{\hfill  {\vrule height6pt width6pt depth0pt}\medskip}

\newtheorem{thm}{Theorem}[section]
\newtheorem{prp}[thm]{Proposition}
\newtheorem{lem}[thm]{Lemma}
\newtheorem{coro}[thm]{Corollary}

\theoremstyle{remark}

\theoremstyle{definition}

\numberwithin{equation}{section}

\begin{document}


\title
{The Strauss conjecture on negatively curved backgrounds
}

\author{Yannick Sire}
\address{Department of Mathematics\\                Johns Hopkins University\\                Baltimore, MD 21218}
\email{sire@math.jhu.edu}

\author{Christopher D. Sogge}
\address{Department of Mathematics\\                Johns Hopkins University\\                Baltimore, MD 21218}
\email{sogge@jhu.edu}

\author{Chengbo Wang}
\address{School of Mathematical Sciences\\                Zhejiang University\\                Hangzhou 310027, China}
\email{wangcbo@zju.edu.cn}
\urladdr{http://www.math.zju.edu.cn/wang}


\thanks{ The second author was supported in part by NSF Grant DMS-1665373 and the first two authors were supported in part by a Simons Fellowship.
The third author was supported in part by 
NSFC 11971428 and
National Support Program for Young Top-Notch Talents.}

\begin{abstract}
This paper is devoted to several small data existence results for semi-linear wave equations on negatively curved Riemannian manifolds. 
We provide a simple and geometric proof of small data global existence for any power $p\in (1, 1+\frac{4}{n-1}]$ for 
the shifted wave equation on hyperbolic space $\H^n$ involving nonlinearities of the form
$\pm |u|^p$ or $\pm|u|^{p-1}u$.   It is based on
the weighted Strichartz estimates of Georgiev-Lindblad-Sogge~\cite{GLS} (or
Tataru~\cite{Tataru}) on  Euclidean space. We also 
prove a small data existence theorem for variably curved backgrounds
which extends earlier
ones for the constant curvature case of
Anker and Pierfelice \cite{MR3254350} and 
Metcalfe and Taylor \cite{MR2944727}.   We also discuss the role of curvature and state a couple of  open problems.
Finally, in an appendix, we give an alternate proof of dispersive estimates of Tataru~ \cite{Tataru} for ${\mathbb H}^3$ and settle a dispute, in his favor,
raised in~\cite{MT} about his proof.  Our proof is slightly more self-contained than the one
in \cite{Tataru} since it does not make use of heavy spherical analysis on hyperbolic space such as the Harish-Chandra $c$-function; instead  it relies only
on simple facts about Bessel potentials.
\end{abstract}



\dedicatory{To Luis Caffarelli on the occasion of his 70th birthday, with admiration and friendship.}

\maketitle

\tableofcontents


\section{Introduction}\label{sec:relate}
As is well-known,
 wave equations on  hyperbolic space $\H^n$, $n\ge 2$, are closely related with
  wave equations on $\R^n \times \R$, see, e.g., Tataru
 \cite{Tataru}.  This paper is devoted to several results concerning global well-posedness for small data on negatively curved Riemannian manifolds. It is well-known fact that small data existence for nonlinear wave equations with power-like nonlinearities is related to the so-called {\sl Strauss conjecture} in $\mathbb R^n$. This paper is three-fold: we provide first a general small data existence result for some range of $p's$ based on the use of Hamiltonian identities and avoiding the heavy machinery of Strichartz estimates;  second, we provide a geometric alternative proof of the optimal global existence on hyperbolic spaces for smooth data. Finally, we provide the same result for rough data, settling in the case $n=3$ a dispute concerning dispersive estimates by Tataru by providing an 
 alternate argument, which is based on almost the same analytic interpolation scheme.

Let $n\ge 2$ and $(M,g)$ be an $n$-dimensional complete  Riemannian manifold, 
and let $\Delta_g$ be the standard Laplace-Beltrami operator on $M$.
The problems under consideration are of the following form for a scalar unknown function $u:\R\times M\to \R$,
 \begin{equation}\label{eq-nlw}
\left\{
\begin{array}{ ll}
(\partial_t^2-\Delta_g+k) u=F_p(u),
     &    t>0, x\in M\\
u(0,x)=\ep u_0(x), \pa_t u(0,x)=\ep u_1(x),      &  (u_0,u_1)\in C^\infty_0(M),
\end{array}\right.
\end{equation}
where $k$ is a constant such that $-\Delta_g+k\ge 0$, $p>1$, and $F_p\in C^1$ behaves like $\pm |u|^p$ or $\pm |u|^{p-1} u$, by which we mean that
\beeq\label{eq-Fp}|F_p(u)|+|u||F'_p(u)|\le C |u|^p,
\eneq
for some constant $C>0$.
 When $M$ is non-compact, 
$C^\infty_0(M)$ denotes the space of smooth functions with compact support. Otherwise $C^\infty_0(M)=C^\infty(M)$.
For these Cauchy problems, the task is then to determine the range of $p$ such that we have the following small data global existence:
for any given data $(u_0, u_1)$, there exists an $\ep_0>0$ such that
there is a global solution to \eqref{eq-nlw} for any $\ep\in (0, \ep_0]$.
We define the critical power $p_c(n)$ as the infimum of $p>1$ such that there is small data global existence.
 
 These problems are of course closely related to the so called Strauss conjecture, when $(M,g)$ is the Euclidean space and $k=0$, with $F_p(u)=|u|^p$.  The first work in this direction is \cite{John79}, where  John determined the critical power, $1+\sqrt{2}$, for the problem when $n=3$, by
proving  global existence results for $p>1+\sqrt{2}$ and
blow-up results for $p<1+\sqrt{2}$.
It was known from Kato \cite{Kato80} that there is no small data global solution in general, for $n=1$ or $1<p<1+2/(n-1)$.
Shortly afterward, Strauss \cite{Strauss81} 
conjectured that the critical power $p_S(n)$ for other dimensions $n\ge 2$ should be the positive root of the quadratic equation $$(n - 1)p^2
- (n + 1)p - 2 = 0.$$
The existence portion of the conjecture was verified in Glassey \cite{Glassey81ex}
($n=2$), Zhou~\cite{Zhou95} ($n=4$), Lindblad-Sogge \cite{LdSo96} ($n\le 8$), and
Georgiev-Lindblad-Sogge \cite{GLS}, Tataru~\cite{Tataru} (all  $n$, $p_S<p\le p_{\rm conf}$),
where 
$$p_{\rm conf}(n)=1+\frac{4}{n-1}$$
is the conformal power.
  The necessity of $p>p_S(n)$
for small data global existence is due to Glassey \cite{Glassey81bu} ($1<p<p_S(2)$),
Sideris \cite{Sideris} ($1<p<p_S(n)$, $n\ge 4$), 
Schaeffer~\cite{Scha85} ($p=p_S(n)$, $n=2, 3$),
and  Yordanov-Zhang \cite{YorZh06}, Zhou \cite{Zh07} ($p=p_S(n)$, $n\ge 4$).
See Wang-Yu \cite{WaYu12rev} or Wang \cite{Wang18} for more references.

 Another model of particular interest is the so-called Klein-Gordon equation on $\mathbb{R}^n$, with $k=m^2>0$: \beeq\label{eq-KG}(\partial_t^2-\Delta+m^2)u=|u|^p.\eneq
In view of the decay rate for the solutions to  the homogeneous  Klein-Gordon equation, it is natural to expect that the critical power to be given by the 
$$p_F(n)=1+\frac{2}{n},$$ which is also known as the Fujita's exponent for the heat equation.
Lindblad-Sogge~\cite{LdSo96-2} proved small data global existence for any $p>p_F(n)=1+2/n$ with $n=1,2,3$. Moreover,
Keel-Tao \cite{KeelTao99} provided an example 
$$(\partial_t^2-\Delta+m^2)u=F(u_t)-|u|^{p-1}u,$$
for which they showed that there is no global solutions for any $1<p\le p_F(n)$.
Here
$F(v)\sim |v|^{p-1} v$ for $|v|\le 1$,
 $F(v)\sim |v|^{q-1} v$ with some $1<q<p$ for $|v|\ge 1$.

However, $p_F$ is not the critical power for \eqref{eq-KG}.
Actually, it is known that it admits global energy solutions (with small energy data) for any energy subcritical powers, that is $p\in (1, 1+4/(n-2))$.   See, e.g., Keel-Tao \cite[pp. 631-632]{KeelTao99}.

 Real hyperbolic spaces, $(\H^n, h)$, are the first examples of rank 1 symmetric spaces of non-compact type and their spherical analysis is to a certain extent very parallel to the one in the Euclidean case. The problem \eqref{eq-nlw} on  hyperbolic spaces with $k=0$ is
 \beeq\label{eq-KG-h}(\partial_t^2-\Delta_h)u=|u|^p,\eneq 
was first considered by Metcalfe and Taylor in \cite{MT}, where they proved small data global existence for $p\in [5/3, 3]$ for dimension $n=3$, by proving improved dispersive and Strichartz estimates. 
Then Anker and Pierfelice \cite{MR3254350} 
proved global existence for 
the problem \eqref{eq-nlw} on  hyperbolic spaces with $k>-\rho^2$,
 \beeq\label{eq-KG-h2}(\partial_t^2-\Delta_h+k)u=|u|^p,\eneq 
 for any $p\in (1, p_{\rm conf}]$ and $n\ge 2$, where $\rho=(n-1)/2$. 
 Soon after, Metcalfe and Taylor  \cite{MR2944727}  gave an alternative proof for $n=3$ with $k=0$.
 This shows that the critical power for this problem is actually $p_c=1$.

Recall that the spectrum of $-\Delta_h$ is $[\rho^2, \infty)$.  See, e.g.,  McKean~\cite{MR0266100}. 
This means
 that the equation \eqref{eq-KG-h2} is more like a nonlinear Klein-Gordon equation instead of a nonlinear wave equation. Thus, at least heuristically, it is not so surprising that we have small data global existence for any $p>1$ (with a certain upper bound on $p$ for technical reasons).
Actually, in the following general theorem, we prove that there is
small data global existence for any $1<p< 1+2/(n-2)$ (which is understood to be $p\in (1,\infty)$ for $n=2$).
\begin{thm}\label{thm-generic}
Let (M, g) be a smooth, complete Riemannian manifold of dimension $n\ge2$ with
Ricci curvature bounded from below and $\inf_{x\in M}{\rm Vol}_g(B(x))>0 $, where
${\rm Vol}_g(B(x))$
denotes the volume of the geodesic ball of center $x$ and radius $1$ with respect to $g$.
Assume that
$k$ is a constant such that ${\rm Spec}(-\Delta_g+k)\subset (c,\infty)$ for some $c>0$.
Then
 for any 
$p\in (1, 1+2/(n-2))$,
 there exists a constant $\ep_0>0$ such that \eqref{eq-nlw} with $\ep \in (0, \ep_0]$ admits global solutions, provided that the data 
$(u_0, u_1)$ satisfy
\beeq\label{eq-data}\|\sqrt{k-\Delta_g} u_0\|_{L^2(M)}^{2}+
\|u_1\|_{L^{2}(M)}^{2}\le 1.\eneq
\end{thm}
Our proof is elementary and completely avoids the somewhat
delicate dispersive and Strichartz estimates used in the aforementioned earlier works.
We first note that it is easy to prove local well-posedness in $C_tH^{1}\cap C_t^1 L^{2}$ for $p\in (1, 1+2/(n-2))$ by  classical energy arguments. Then the basic observation is that
the problem \eqref{eq-nlw}  is Hamiltonian and, for these types of ``Klein-Gordon  equations'',   the nonlinear part can be easily controlled by the linear part. Such arguments are also classical (see, e.g.,
Cazenave\cite{MR780103}, Keel-Tao \cite{KeelTao99}).
We remark also that the assumptions on Ricci curvature and ${\rm Vol}_g (B(x))$ are made to ensure the Sobolev estimates
\beeq\label{eq-Sobo}\|f\|_{L^{q}(M)}\les \|\sqrt{k-\Delta_g} f\|_{L^2(M)},  \quad 2\le q\le 2n/(n-2),\eneq
where it is understood that
$q\in [2,\infty)$ when $n=2$.

As a simple application, we see that
Theorem~\ref{thm-generic} applies for any manifold $(M,g)$, with $k>0$, since $-\Delta_g$ is nonnegative. The condition on $k$ is sharp in general, as we have seen that the critical power $p_c=p_S(n)>1$
for the Strauss conjecture on $\R^n$ ($k=0$).
The worst situation occurs for compact manifolds, for which it is easy to see that, generically, there is no small data global existence results for 
 \eqref{eq-nlw} with $k=0$ for any $p>1$. Actually, the simplest examples for this are \eqref{eq-nlw} with $F_p(u)=\pm |u|^p$ or $|u|^{p-1}u$ and constant data, which reduces to the ODE $u_{tt}=F_p(u)$ 
 which has the property that 
 generic solutions blow up in finite time. In particular, there is no small data global existence for 
 \eqref{eq-nlw} with $k=0$ and $F_p(u)=\pm |u|^p$ or $|u|^{p-1}u$, for any  complete Riemannian 
 manifolds $(M,g)$ with positive lower bound on the Ricci curvature,
 which is compact due to the Bonnet-Myers theorem (see e.g. \cite[p. 84]{Chavel}).

An important class of manifolds with the property ${\rm Spec}(-\Delta_g)\subset (0,\infty)$ is a
simply connected, complete, Riemannian $n$-manifold with sectional curvature $K\leq -\kappa$ for some constant $\kappa>0$, 
for which it is known that ${\rm Spec}(-\Delta_g)\subset [\rho^2\kappa,\infty)$,
see \cite{MR0266100}.
Recall also that a lower bound of 
 sectional curvature implies
that for the Ricci curvature and that an upper bound
ensures that $\text{Vol}_g(B(x))>\delta>0$ for some $\delta>0$ by the G\"unther comparison
theorem (see e.g. \cite[p. 129]{Chavel}).  Consequently, Theorem~\ref{thm-generic} yields the following
result for simply connected complete manifolds with negatively pinched curvature:

\begin{coro}\label{thm-NegCurv}
Let $(M,g)$ be a 
simply connected, complete, Riemannian manifold of dimension $n\ge2$ with sectional curvature $K\in [-\kappa_2, -\kappa_1]$ for some $\kappa_2\ge \kappa_1>0$.
Then for any $k>-\rho^2 \kappa_1$ and $p\in (1, 1+2/(n-2))$, there exists a constant $\ep_0>0$ such that the problem
 \beeq\label{eq-KG-NegCurv}(\partial_t^2-\Delta_g+k)u =F_p(u), \, \,  u(0)=\ep u_0,\ u_t(0)=\ep u_1\ \eneq 
 with  $\ep \in (0, \ep_0]$ admits global solutions, provided that the data 
$(u_0, u_1)$ satisfy
\eqref{eq-data}.
\end{coro}

We remark that this Corollary could be strengthened a bit by using, say, the results in \cite{MR1043421} and \cite{MR1353612} which involve slightly weaker curvature assumptions that also ensure that
$\text{Spec }(-\Delta_g+k)\subset (c,\infty)$, some $c>0$.

To state another corollary recall  that $(M,g)$ is said to be asymptotically hyperbolic, in the sense of Mazzeo-Melrose \cite{MM}, if there is a compact Riemannian manifold with boundary $(X, \tilde g)$, such that $M$ could be realized as the interior of $X$, with metric 
$g=f^{-2}\tilde g$, where $f$ is a smooth boundary defining function
\footnote{Here $f\ge 0$ on $X$, $\pa X=f^{-1}(0)$, and $df \neq 0$ on $\pa X$. } 
with 
$\|df\|_{\tilde g}=1$ on $\pa X$. It is known\footnote{The third author would like to thank Fang Wang and Meng Wang for helpful information on the spectrum.
}
that $${\rm Spec}(-\Delta_g)=[\rho^2,\infty)\cup \sigma_{pp}, \, \, \, 
\sigma_{pp}\subset (0, \rho^2),
$$
where  the pure point spectrum, $\sigma_{pp}$ (the set of $L^2$ eigenvalues), is finite.
See, e.g., 
 Graham and Zworski 
\cite[page 95-96]{MR1965361}. In particular, we see that ${\rm Spec}(-\Delta_g)\subset (c,\infty)$ for some $c>0$ and so Theorem \ref{thm-generic} applies with $k=0$ in this setting.  Consequently we have the following:
\begin{coro}\label{thm-AH}
Let $(M,g)$ be an $n$-dimensional  asymptotically hyperbolic manifold.
Then the problem
 \beeq\label{eq-KG-AH}(\partial_t^2-\Delta_g)u=F_p(u), \, \, u(0)=\ep u_0,\ u_t(0)=\ep u_1\ \eneq 
admits small data global solutions for any 
$p\in (1, 1+2/(n-2))$.
\end{coro}


As we see from Theorem \ref{thm-generic}, in the Klein-Gordon case, the problem is relatively simple and the machinery of the Strichartz estimates could be avoided. As we have seen from the Strauss conjecture, the case of wave equations is much more delicate. 
To handle this case, one expects to have to develop space-time estimates that are specifically well-adapted to the problem.

In the case of the existence problem on hyperbolic spaces, 
 that is,
\eqref{eq-nlw} with $(M,g)=(\H^n, h)$ and 
$k=-\rho^2$, $\rho=(n-1)/2$,
 such that ${\rm Spec}(-\Delta_h+k)=[0,\infty)$,
 \beeq\label{eq-wave-h}\Box_{\H^n} u:=(\partial_t^2-\Delta_h-\rho^2)u=F_p(u), \quad
u(0)= \ep u_0, \, \, \pa_t u(0)=\ep u_1,
 \eneq  we expect that
 the critical power $p_c(n)$ satisfies
  $p_c(n)\le p_S(n)$, due to negative curvature and the resulting better decay behavior for the linear waves.
For convenience of presentation, we set $D_0=\sqrt{-\Delta_{\H^n}-\rho^2}$, $D=\sqrt{-\Delta_{\H^n}}$ and then we have
$$\Box_{\H^n}=\pt^2+D_0^2.$$

It was considered earlier by Fontaine
\cite{Fo97}, where 
 global existence with small data was proved for $n=2, 3$ and $p\ge 2$.
We note that
Anker, Pierfelice and Vallarino
\cite{APV12}, \cite{MR3345662} proved dispersive and Strichartz estimates 
for linear (shifted) wave equations 
on hyperbolic spaces and more generally Damek-Ricci spaces, which behave better than ones in  Euclidean space. 
With help of these estimates, it was shown that there is small data global existence
for certain $p > 1$ arbitrarily close to 1, which shows that the critical power $p_c(n)=1$.

The second aim of the present work is to provide a simple geometric proof of the small data global existence for
the less favorable equation \eqref{eq-wave-h} with
any power $p\in (1, 1+4/(n-1)]$. 
More precisely, we will prove the following result, based on the
space-time weighted Strichartz estimates of 
Georgiev-Lindblad-Sogge \cite{GLS} (see also Tataru \cite{Tataru} for the scale-invariant case).
\begin{thm}\label{GWPsmooth}
Let $p\in (1, p_{\rm conf}]$. Assume further that $F_{p}(u)$ is a homogeneous function of $u$, of order $p$, i.e.,  $F_{p}(u)=c|u|^{p-1}u$ or $c|u|^{p}$ for some $c$.
Then, for any $(u_0, u_1)\in C^\infty_0$, there exists a constant $\ep_1>0$ such that \eqref{eq-wave-h} with 
$\ep \in (0, \ep_1]$ admits  global solutions.
\end{thm}

 As already mentioned, the spherical analysis on $\H^n$ is very similar to the one of $\mathbb R^n$. Here we provide a very simple geometric argument based on the fact that, on real hyperbolic space, the conformal Laplacian is conformally covariant and that $\H^n$ is conformal to $\R^n$. Of course, this argument does not work, as far as we know, for other rank one symmetric spaces of non-compact type, and even less on Damek-Ricci spaces (for which the spherical analysis is actually similar to the one of the hyperbolic space).

In the statement of Theorem \ref{GWPsmooth}, we assume the data to be smooth with compact support. As usual, with some more effort, we could relax the condition to admit less regular data.   Specifically, we have the following:

\begin{thm}\label{GWPgldata}
Under the same assumption as in  Theorems \ref{GWPsmooth}.
There exists a constant $\ep_2>0$ such that  \eqref{eq-wave-h}  with $\ep \in (0, \ep_2]$ admits  global solutions for any $(u_0, u_1)$, provided that
\beeq\label{eq-data-h}\|D^s u_0\|_{L^{(p+1)/p}(\H^n)}+
\|D^{s-1}u_1\|_{L^{(p+1)/p}(\H^n)}\le 1\ ,\eneq
where $s=(n+1)(\frac{1}{2}-\frac{1}{p+1})$.
\end{thm}
 As we have mentioned, the condition on the data 
 could be replaced by using the Strichartz estimates (see, e.g., \cite{APV12}) to conditions that 
 \beeq\label{eq-data3-h}\|
 D^{s/2-1/2}D_0^{1/2}
  u_0\|_{L^2(\H^n)}+
\| D^{s/2-1/2}D_0^{-1/2}u_1\|_{L^{2}(\H^n)}\le 1.\eneq
 Here, instead of directly using Strichartz estimates, we present a proof, based on
the dispersive estimates 
of Tataru \cite{Tataru}
for the linear homogeneous waves on hyperbolic space. See also
\eqref{eq-data4-h} for alternative conditions on the data.

In Theorems \ref{GWPsmooth} and 
\ref{GWPgldata}, we have restricted ourselves to the conformal or sub-conformal case, $p\le p_{\rm conf}$. As usual, the idea of proof could be further exploited to prove results for certain larger powers. 
We thank the referee for drawing our attention to \cite{MR3345662}, 
where global existence for certain super-conformal powers has been discussed.
 Here, as illustration, we present a stronger result in the following
\begin{thm}\label{thm-superconformal}
Let $p\in (p_{\rm conf}, 1+\frac{4n}{n^2-3n-2}]$ (which is understood to be $p\in (p_{\rm conf}, \infty)$ for $n=2,3$) and $s=
2-\frac{1}{n+1}+\de$, with  $\de\in (0, \frac 2{(n^{2}-1)(p-1)})$.
Then there exists a constant $\ep_3>0$ such that  \eqref{eq-wave-h}  with $\ep \in (0, \ep_3]$ admits  global solutions for any $(u_0, u_1)$, provided that
\beeq\label{eq-data-h2}\|D^s u_0\|_{L^{2/(1+2\de)}(\H^n)}+
\|D^{s-1}u_1\|_{L^{2/(1+2\de)}(\H^n)}\le 1.\eneq
\end{thm}

\medskip
\noindent {\bf Remark.}
 In \cite{MR3345662}, global results were obtained for any $1<p<p_{1}(n)$, where
 $$p_{1}(n)=\left\{\begin{array}{ll }
5/2   &   n=4,\\
\frac{6+\sqrt{21}}5  &   n=5,\\
1+\frac{2}{(n-1)/2-1/(n-1)}      &   n\ge 6    .
\end{array}\right.$$
  For comparison, if $n\ge 4$, and let $p_2(n)=1+\frac{4n}{n^2-3n-2}$ be as in Theorem \ref{thm-superconformal}, then $p_{1}(n)<p_{2}(n)$, which means that our results improve those in 
  \cite{MR3345662}  somewhat.  
Moreover,
there appears to be gaps in the proof of the super-conformal result given there. 
 For example,
(49) on 
  \cite[page 751]{MR3345662} for the case 
$n=6$, $\gamma=2$ (which is $p$
in our notation), could not be satisfied with their choice of $q=14/3$ (see page 752, Case (D)), as 
$$\frac{\ga}2+\frac{n-5}{2(n-1)}-\frac{\ga}{q}=1+\frac{1}{10}-\frac{3}{7}>\frac{1}{2}\ .$$

\medskip

\subsubsection*{Outline} Our paper is organized as follows. In the next section, we
present the proof of global existence for Klein-Gordon type equations, Theorem \ref{thm-generic}, for fairly general manifolds. 
In \S\ref{sec:Strau}, we recall the relation between the wave equations on  hyperbolic space $\H^n$, and Euclidean space
 and
prove 
global existence results
for wave equations on  $\H^n$,
with $C_0^\infty$ data,
Theorem \ref{GWPsmooth}, by using 
 the space-time weighted Strichartz estimates of 
Georgiev-Lindblad-Sogge \cite{GLS} and Tataru \cite{Tataru}.
In 
\S\ref{sec:generaldata}
we prove Theorem~\ref{GWPgldata}, by
removing the restriction of compact  support and 
relaxing the regularity condition on the initial data imposed in Theorem \ref{GWPsmooth}. The idea is to exploit the dispersive estimates
of Tataru \cite{Tataru},
 for the linear homogeneous waves on hyperbolic spaces.
In addition, in \S\ref{sec:altprf}, an alternate proof of  Theorem \ref{GWPgldata} for 
 $p\in (1, 1+2/(n-1))$, as well as another global result involving different
 conditions on the data, \eqref{eq-data4-h}, 
  are obtained  after proving certain Strichartz type estimates. 
In \S\ref{sec:superconformal}, we present the proof of 
Theorem \ref{thm-superconformal}.
   Lastly, in an Appendix, we give
  an independent proof of dispersive estimates of Tataru~\cite{Tataru} for ${\mathbb H}^3$ and
  explain how there is an incorrect assertion in \cite{MT} that there is a gap in Tataru's argument.

\section{Global existence for Klein-Gordon type equations on manifolds}
In this section, we shall present the proof of global existence for Klein-Gordon type equations, Theorem \ref{thm-generic}.

First, though, let us present the Sobolev estimates that we shall require.
\begin{lem}\label{thm-Sobo}
Let $\|f\|_{H^1}=\|\sqrt{k-\Delta_g} u_0\|_{L^2(M)}$ be the natural Sobolev norm for the positive operator $k-\Delta_g$,  then we have the Sobolev estimates
\eqref{eq-Sobo}.
\end{lem}

\begin{proof}
As $k-\Delta_g>0$, we know from spectral theorem that
$$\|f\|_{L^2(M)}+\|\sqrt{-\Delta_g} f\|_{L^2(M)}\le C\|f\|_{H^1(M)},$$
for some constant $C>0$. Here we see that the left hand side is just the standard $H^1$ norm on $(M,g)$, for which the standard Sobolev embedding is available,
for smooth complete manifolds with Ricci curvature bounded from below and $\inf_{x\in M}{\rm Vol}_g(B(x))>0$.
 See, e.g., Hebey \cite[Theorem 3.2]{MR1688256} for $n\ge 3$. When $n=2$, the result $H^1\subset L^q$ for any $q\in [2,\infty)$ could be derived from 
\cite[Theorem 3.2]{MR1688256} with $q=1$ 
using a similar argument in \cite[Lemma~2.1]{MR1688256}.
\end{proof}

\begin{proof}[Proof of Theorem~\ref{thm-generic}]
If we let
$$E(t)=\|u_t\|_{L^2}^2+\|u\|_{H^1}^2$$
be the energy functional, we see that, if $u_{tt}-\Delta_g u+ku=F$,  then
\begin{align*}
\frac{d}{dt}E(t)=2\<u_t, u_{tt}\>+2\<\sqrt{k-\Delta_g} u, \sqrt{k-\Delta_g}u_{t}\>
&=2\<u_{t}, u_{tt}-\Delta_g u+ku\>
\\
&=2\<u_{t}, F\>\le 2E^{1/2}\|F\|_{L^{2}}.
\end{align*}
This yields the natural energy estimates for $t\ge 0$
\beeq\label{eq-energy}E(t)^{1/2}\le E(0)^{1/2}+\int_{0}^{t} \|F(\tau)\|_{L^{2}}d\tau.\eneq

With help of the Sobolev embedding and energy estimates, we are able to prove local well-posedness in $C H^1\cap C^1 L^2$.
Observe that for any given $p\in (1, 1+2/(n-2))$, we know
from H\"older's inequality,
 the Sobolev embedding ($H^1\subset L^{2p}$) and \eqref{eq-Fp} that there exist constants $C_{1}$ and $C_{2}$
 such that \begin{align*}
 \|F_{p}(u)-&F_{p}(v)\|_{ L^{1}([0,T]; L^{2})}
 \\ &\le C_{1}T
 (\|u\|_{L^{\infty}([0,T]; L^{2p})}+\|v\|_{L^{\infty}([0,T]; L^{2p})})^{p-1}
 \|u-v\|_{L^{\infty}([0,T]; L^{2p})} 
 \\
&\le  C_{2}T ( \|u\|_{C([0,T]; H^{1})}^{p-1}+ \|v\|_{C([0,T]; H^{1})}^{p-1})
 \|u-v\|_{C([0,T]; H^{1})},
\end{align*}
for any
 $u, v\in C([0,T]; H^1)\cap C^1([0,T]; L^2)\subset L^{\infty}([0,T]; L^{2p})$.
Combined with \eqref{eq-energy}, a standard contraction mapping argument yields local well-posedness for \eqref{eq-nlw} in $C([0,T_{*}); H^1)\cap C^1([0,T_{*}); L^2)$, for some
$$T_{*}\ge \frac{ E(0)^{-(p-1)/2}}{2^{p+1}C_{2}}
\ge 
\frac{ \ep^{-(p-1)/2}}{2^{p+1}C_{2}},$$
where we have used the assumption \eqref{eq-data}.
Moreover, if $T_{*}$ is the maximal time of existence, with $T_{*}<\infty$, we have
$$\sup_{t\in [0, T_{*})} E(t)=\infty.$$

To prove the theorem, it remains to give a uniform a priori control on the energy of the solution, for small $\ep$.
Observe that the problem \eqref{eq-nlw} is Hamiltonian with the Hamiltonian functional given by
$$H[u(t), u_{t}(t)]=\int \left(\frac{u_t^2+|\sqrt{k-\Delta_g} u|^2}{2}-G_p(u)\right)dV_g,$$
where 
$G_p$ is the primitive function of $F_p$ with $G_p(0)=0$,
and $dV_{g}$ is the standard volume form for $(M,g)$.
Applying this fact to the solution $u\in C([0,T_{*}); H^1)\cap C^1([0,T_{*}); L^2)$ for \eqref{eq-nlw}, we see that
\beeq\label{eq-Hamilton}H[u(t), u_{t}(t)]= H[\ep u_{0}, \ep u_{1}]\le C_{3}\ep^{2}, \quad \forall t\in [0, T_{*}),\eneq
for some $C_{3}>0$ and any $\ep\le 1$.
Then we have
\begin{eqnarray*}E(t)&
=& 2H[u(t), u_{t}(t)]+2\int G_p(u)dV_g\\
&\le &
2H[u(t), u_{t}(t)]+C\|u\|_{L^{p+1}}^{p+1}\\
&\le& 2
H[u(t), u_{t}(t)]+\tilde C\|u(t)\|_{H^1}^{p+1}\\
&\le& 2C_{3}\ep^{2}+C_4E(t)^{(p+1)/2},
\end{eqnarray*}
where we have used the fact that $|G_p(u)|\le C |u|^{p+1}/(p+1)$,
by \eqref{eq-Fp}. Therefore, a continuity argument implies that
\beeq\label{eq-apriori} E(t)\le 4 C_{3}\ep^{2}, \quad \forall t\in [0, T_{*}),\eneq
as long as
\beeq\label{eq-datasize}\ep\le \ep_{0}:=(4C_{4})^{-1/(p-1)}(4C_{3})^{-1/2}.\eneq

In view of the local well-posed results, we see that \eqref{eq-apriori} is sufficient to conclude $T_{*}=\infty$ and so is the proof of global existence with $\ep\in (0, \ep_0]$, where $\ep_{0}$ is given by \eqref{eq-datasize}.
\end{proof}

\section{The Strauss conjecture on hyperbolic space}\label{sec:Strau}

In this section, we first recall the relation between the wave equations on the hyperbolic space-time $\H^n \times \R$, $n\ge 2$, and
 the wave equations on $\R^n \times \R$. With help of this fact,
we present the proof of Theorem \ref{GWPsmooth}, by using 
 the space-time weighted Strichartz estimates of 
Georgiev-Lindblad-Sogge \cite{GLS} and  Tataru \cite{Tataru}.

Recall that inside the forward light cone, $\La=\{(x, t)\in \R^n \times \R: t>|x|\}$, we may introduce 
coordinates
$$r=|x|, \, \, \, t=e^{\tau}\cosh s, \, \, \, r=e^{\tau}\sinh s,\, \,  \, s\in [0,\infty), \, \, \, \tau\in \R.$$
Here, with $\omega\in\Sp^{n-1}$, we may view $(s, \omega)$,  as  natural polar coordinates in  hyperbolic space $\H^n:=\La_{\tau=0}$, with the natural metric, $ds^2+
(\sinh s)^2 d\omega^2$,
 induced from the Minkowski metric $g=-dt^2+dx^2=
-dt^2+dr^2+r^2d\omega^2$ to $\H^n$. In the new coordinates,  a simple computation leads to
$$\Box=-\pt^2+\Delta 
=e^{-2\tau}(-\pa_\tau^2+\Delta_{\H^n}-(n-1)\pa_\tau)
=e^{-(2+\rho)\tau}(-\pa_\tau^2+\Delta_{\H^n}+\rho^2)e^{\rho \tau},
$$
with $\rho=(n-1)/2$.
That is, with 
$\Box_{\H^n}:=-\pa_\tau^2+\Delta_{\H^n}+\rho^2$, we have
\beeq\label{eq-1.1}\Box=
e^{-(2+\rho)\tau} \Box_{\H^n}e^{\rho \tau}.
\eneq

Let $u_0, u_1\in C^\infty_0(\H^n)$ 
and consider the Cauchy problem
\eqref{eq-wave-h} with $p>1$ and
  small data $(\ep u_{0}, \ep u_{1})$.
By \eqref{eq-1.1}, we know this problem is equivalent to solving, with $u=e^{\rho \tau} w$, 
\begin{multline}\label{eq-2.2}\Box w=e^{-(2+\rho)\tau} \Box_{\H^n}e^{\rho \tau} w = e^{-(2+\rho)\tau} \Box_{\H^n} u
=e^{-(2+\rho)\tau} F_{p}(u)
\\
=e^{-(2-\rho (p-1))\tau}F_{p}(w)
=(t^2-r^2)^{-\sigma }F_{p}(w)
\end{multline}
with  $C^\infty_0$ data of form $\ep(w_0, w_1)$ on $t=\sqrt{1+r^2}$,
where we have use the assumption that $F_{p}$ is homogeneous and
$$\sigma=1-\frac{\rho}2(p-1).$$

To solve the Cauchy problem \eqref{eq-2.2}, we recall two facts about wave equations.
The first is a weighted Strichartz estimates of
Tataru
 \cite[Theorem 5]{Tataru}. See also Georgiev-Lindblad-Sogge \cite[Theorem 1.2]{GLS} for an earlier version, which is sufficient to prove results for compactly supported data.
\begin{lem}[Weighted Strichartz estimates]
Let $n\ge 2$ and $w$ be a solution of the equation $\Box w =F$ which is supported inside the forward
light cone. Then the following estimate holds:
\beeq\label{eq-2.3}\|(t^2-r^2)^{\ga_1} w\|_{L^q(\R^{n+1})}\le C_{q,\ga_1,\ga_2}
\|(t^2-r^2)^{\ga_2} F\|_{L^{q'}(\R^{n+1})},\eneq
provided that $2\le  q\le   2(n + 1) /(n - 1)$ and
$$
\ga_1<\frac{n-1}{2}-\frac{n}{q},\ 
\ga_2=\ga_1-
\frac{n-1}{2}+\frac{n+1}{q}.$$
\end{lem}

In addition, it is well-known that the solutions of the homogeneous wave equation with compactly supported smooth data, of size $\ep$, satisfy 
\beeq\label{eq-2.4}|w(t,x)|\les \ep (t^2-r^2)^{-\rho}.\eneq

With help of \eqref{eq-2.3} and \eqref{eq-2.4}, it is not hard to show that \eqref{eq-2.2} admits global solutions for any 
$p\in (1, p_{\rm conf}]$.
Actually, by setting $q=p+1$ and $\ga_2=-\ga_1=\frac{\sigma}{p+1}$,
such that we have
$$
\ga_{2}-\sigma=\ga_{1}p,\  \ga_{2}=\ga_{1}-
\frac{n-1}{2}+\frac{n+1}{q},$$
 we can solve \eqref{eq-2.2} by iteration. Let $w^{(-1)}=0$, we define inductively 
\beeq\label{eq-2.5}\Box w^{(m)}
=(t^2-r^2)^{-\sigma }|w^{(m-1)}|^p, \quad m=1,2,\dots,
\eneq
with  $C^\infty_0$ data
$\ep(w_0, w_1)$
 on $t=\sqrt{1+r^2}$. 
 
Let
$\|w\|_X:=\|(t^2-r^2)^{\ga_1} w\|_{L^q(t>\sqrt{1+r^2})}$, then, by a routine calculation, we see from
 \eqref{eq-2.4}
that
\begin{equation}\label{3.66}\|w^{(0)}\|_X\le C_0 \ep.
\end{equation}
As a result by  \eqref{eq-2.3}, for $m\ge 1$ we have
$$\|w^{(m)}\|_X\le 
\|w^{(0)}\|_X+\|w^{(m)}-w^{(0)}\|_X
\le
C_0 \ep
+C_1\|w^{(m-1)}\|_X^p.$$
Based on these estimates, a standard continuity argument ensures that
$$\|w^{(m)}\|_X
\le
2 C_0 \ep,$$
provided $\ep\le \ep_0$ with $\ep_0\ll 1$.
Moreover,
with a possibly smaller $\ep_1\ll 1$, we have the
convergence of $w^{(m)}$  in $X$, which proves the global existence of weak solutions for
 \eqref{eq-2.2},
with
 sufficiently small data of size $\ep\le \ep_1$.
This completes the proof of Theorem \ref{GWPsmooth}.

\section{General data: proof of Theorem \ref{GWPgldata}}\label{sec:generaldata}

In this section,
we present a proof
of the Strauss conjecture on hyperbolic spaces with general data,
 Theorem \ref{GWPgldata}, based  on
the dispersive estimates  of Tataru \cite{Tataru}
for the linear homogeneous wave equation
on hyperbolic spaces. 
In addition, an alternative proof of  Theorem \ref{GWPgldata} for 
 $p\in (1, 1+2/(n-1))$ will be presented in
\S \ref{sec:altprf}, by proving certain Strichartz type estimates.

From the proof of global results in Section \ref{sec:Strau}, we see that we need only to ensure the first iteration $w^{(0)}\in X$, for general data.
To achieve this goal, we would like to translate it back to  hyperbolic space.

Observing that
\begin{eqnarray*}
\|(t^2-r^2)^{\ga_1} w\|_{L^q(\La)}^q & = & \int (t^2-r^2)^{\ga_1 q}|w|^q dt dx\\
 & = & \int e^{2\ga_1 q\tau}|w|^q e^{(n+1)\tau}d\tau dV_{\H^n}\\
  & = & \int e^{(2\ga_1-\rho) q\tau}|u|^q e^{(n+1)\tau}d\tau dV_{\H^n}\\
    & = & \| e^{(2\ga_1-\rho+\frac{n+1}{q}) \tau}u  \|_{L^q(d\tau dV_{\H^n})}^q
,
\end{eqnarray*}
we see that what we need is to find an estimate for
\beeq\label{eq-2.6}
 \| e^{(2\ga_1-\rho+\frac{n+1}{q}) \tau}u^{(0)}  \|_{L^q(d\tau dV_{\H^n})},
\eneq
 where
$\Box_{\H^n} u^{(0)}= 0$ with
data
$\ep (u_0, u_1 )$
 on $\tau=0$.
 Recall that
 $$ u^{(0)}=\ep C(\tau)u_0+\ep S(\tau)u_1,$$
where  $S(\tau)=D_0^{-1}\sin (\tau D_0)$,
  $C(\tau)=\cos (\tau D_0)$, with $D_0=\sqrt{-\Delta_{\H^n}-\rho^2}$.

 To control \eqref{eq-2.6}, we recall the following dispersive estimate of
Tataru
 \cite[Theorem 3]{Tataru}: 
\begin{lem}[Dispersive estimates]\label{thm-dispersive}
Let $D=\sqrt{-\Delta_{\H^n}}$. Then the following estimate holds:
\beeq\label{eq-2.7}
\|S(\tau)f\|_{L^q}\les 
\frac{(1+\tau)^{\frac 2q}}{(\sinh \tau)^{(n-1) (\frac{1}{2}-\frac{1}{q})}}
\|D^{(n+1)(\frac{1}{2}-\frac{1}{q})-1}f\|_{L^{q'}},  \quad 2\le q<\infty,
\eneq
\beeq\label{eq-2.8}
\|C(\tau)f\|_{L^q}\les 
\frac{1}{(\sinh \tau)^{(n-1) (\frac{1}{2}-\frac{1}{q})}}
\|D^{(n+1)(\frac{1}{2}-\frac{1}{q})}f\|_{L^{q'}},  \quad 2\le q<\infty.
\eneq
\end{lem}

If we recall that
$\ga_1=-\sigma/q$, $q=p+1$,
$\sigma=1-\rho(p-1)/2$, $\rho=(n-1)/2$, 
it is easy to check that \beeq\label{eq-2.9} 2\ga_1-\rho+\frac{n+1}{q}<
(n-1) \left(\frac{1}{2}-\frac{1}{q}\right)
\Leftrightarrow
\ga_1+\frac{n}{p+1}<\rho\Leftrightarrow
p>1
. 
\eneq 
Consequently we see from Lemma~\ref{thm-dispersive} that for any $q=p+1>2$ there exists a constant $C>0$ such that 
$$\|w^{(0)}\|_X=\| e^{(2\ga_1-\rho+\frac{n+1}{q}) \tau}u^{(0)}  \|_{L^q(d\tau dV_{\H^n})}
\le C \ep
(\|D^su_0\|_{L^{\frac{p+1}p}}+
\|D^{s-1}u_1\|_{L^{\frac{p+1}p}}),
$$
with $s=(n+1)(\frac{1}{2}-\frac{1}{p+1})$.
Since this is \eqref{3.66}, we obtain Theorem \ref{GWPgldata} as before.

\medskip
\noindent  {\bf Remark.}  It was pointed out by the referee that Metcalfe and Taylor~\cite{MT} assert that when 
$n=3$ the proof in Tataru~\cite{Tataru} of \eqref{eq-2.7}--\eqref{eq-2.8} has a gap.  As a result, we felt the
need to present an independent proof of these dispersive estimates for ${\mathbb H}^3$.  Our proof makes use
of elementary properties of Bessel potentials.  Unlike the proof in \cite{Tataru}, it is a bit more self-contained in the sense that it does not rely on the use of the Harish-Chandra $c$-function.

 The claim in \cite{MT} rests on the fact that Bessel functions
of order one have a logarithmic singularity at the origin (and blow up logarithmically at infinity).  We shall present our proof of Tataru's dispersive estimates  \eqref{eq-2.7}--\eqref{eq-2.8}  for $n=3$ in such a way that we can highlight the oversight in \cite{MT} which lead the authors to make their incorrect assertion about this ``gap''.  We thank the referee for bringing this issue to our attention so that we may hopefully settle this minor simple misunderstanding
and use these dispersive estimates.  We are also grateful to Michael Taylor for helpful comments.

\section{Strichartz type estimates and an alternative proof of global existence for the shifted wave equation on 
$\H^n$}\label{sec:altprf}
As a side remark, let us now show how we could  use Lemma \ref{thm-dispersive} to prove  inhomogeneous Strichartz type estimates that 
 are sufficient to 
 give an alternative proof of  Theorem \ref{GWPgldata} for $p\in (1, 1+2/(n-1))$.

First, let us observe that when
 $p\in (1, p_{\rm conf}]$ we have $s\le 1$, and so
 \beeq\label{eq-2.7'}
\|S(\tau)f\|_{L^{p+1}(\H^n)}\les 
\, K_{p+1}(\tau)\|f\|_{L^{\frac{p+1}p}(\H^n)},
\eneq
by \eqref{eq-2.7},
where \beeq\label{eq-5.2} K_q(\tau)=\frac{(1+\tau)^{\frac 2q}}{(\sinh \tau)^{(n-1) (\frac{1}{2}-\frac{1}{q})}}.\eneq

Based on  Duhamel's principle and \eqref{eq-2.7'}, we see that 
$$\|u(\tau)\|_{L^{p+1}(\H^n)}\le C
\int_{0}^{\tau}
K_{p+1}(\tau-s)
\|F(s)\|_{L^{\frac{p+1}p}(\H^n)}d s$$
for  solutions to
$\Box_{\H^n} u=F$ with vanishing data at $\tau=0$.
Since 
$$K_{p+1}(\tau)\chi_{\tau>0}\les |\tau|^{-2/(p+1)}, \quad 1<p\le p_{\rm conf},$$ 
we obtain the inhomogeneous Strichartz estimates 
\beeq\label{eq-Strichartz}\|u\|_{L^{p+1}(\R^+\times\H^n)}\le C
\|F\|_{L^{(p+1)/p}(\R^+\times\H^n)},\ 1<p\le p_{\rm conf},
\eneq
 by the Hardy-Littlewood-Sobolev inequality. 

Concerning the homogeneous solutions, we  observe that
if $$(n-1) \left(\frac{1}{2}-\frac{1}{p+1}\right)<\frac 1{p+1},$$
that is, $1<p< 1+\frac{2}{n-1}$, we have
$K_{p+1}(\tau)\in L^{p+1}(\R^+)$ and so
we have the homogeneous estimates 
$$\|u\|_{L^{p+1}(\R^+\times\H^n)}\les
\|D^s u(0)\|_{L^{(p+1)/p}(\H^n)}+
\|D^{s-1}u_\tau (0)\|_{L^{(p+1)/p}(\H^n)}
$$
for any solutions to 
$\Box_{\H^n} u=0$,
in view of Lemma \ref{thm-dispersive},
where $s=(n+1)(\frac{1}{2}-\frac{1}{p+1})$.

For general $p$,
the argument still works, if we impose  other conditions on the data.
To state these we require the following homogeneous estimates.
\begin{lem}[Homogeneous estimates]\label{thm-homoStri}
Let $q\in [1,\infty)$ and $r\in (2,\infty)$, then for any 
$r_0\in (2, r]$ such that 
\beeq\label{eq-condi}\frac 1q>(n-1)\left(\frac 12-\frac 1{r_0}\right),\eneq
we have
\beeq\label{eq-homoStri}\|u\|_{L^q([0,\infty); L^r(\H^n))}\les 
\|D^{s_0} u(0)\|_{L^{r_0'}(\H^n)}+
\|D^{s_0-1}u_\tau (0)\|_{L^{r_0'}(\H^n)},\eneq
if
$\Box_{\H^n} u=0$ and $s_0=\frac{n+1}{2}-\frac{1}{r_0}-\frac{n}{r}$.
In particular, for any $p\in (1, \infty)$, 
 we have
\beeq\label{eq-homoStri2}\|u\|_{L^{p+1}(\R^+\times\H^n)}\les 
\|D^{s_1} u(0)\|_{L^{\frac{2}{1+2\de}}(\H^n)}+
\|D^{s_1-1} u_\tau(0)\|_{L^{\frac{2}{1+2\de}}(\H^n)}\eneq
with $s_1=n(\frac 12-\frac 1{p+1})+\de$, for any $\de>0$ sufficiently small such that
$\de<\frac{1}{(n-1)(p+1)}$ and $\de\le \frac 12-\frac 1{p+1}$.
\end{lem}
To  prove this lemma, we need only to prove \eqref{eq-homoStri}.
By 
Sobolev embedding and
Lemma~\ref{thm-dispersive}, we have
$$
\|u(\tau)\|_{L^{r}(\H^n)}\les
\|D^{s_2}u(\tau)\|_{L^{r_0}(\H^n)}\les K_{r_0}(\tau)(
\|D^{s_0} u(0)\|_{L^{r_0'}(\H^n)}+
\|D^{s_0-1}u_\tau (0)\|_{L^{r_0'}(\H^n)}),
$$
with  $s_2=n(1/r_0-1/r)$
and $s_0-s_2=(n+1)(1/2-1/r_0)$.
Notice that \eqref{eq-condi} ensures $K_{r_0}\in L^q$, and so we obtain
\eqref{eq-homoStri}, which completes the proof of
Lemma \ref{thm-homoStri}.

With help of \eqref{eq-homoStri2} and \eqref{eq-Strichartz}, it is  standard to conclude the proof of Theorem~\ref{GWPgldata}, 
with the condition on the data replaced by
\beeq\label{eq-data4-h}\|D^{s_1} u_0\|_{L^{\frac{2}{1+2\de}}(\H^n)}+
\|D^{s_1-1} u_1\|_{L^{\frac{2}{1+2\de}}(\H^n)}\le 1,\eneq
where $s_1=n(\frac 12-\frac 1{p+1})+\de$, and $\de>0$ is sufficiently small such that
$\de<\frac{1}{(n-1)(p+1)}$ and $\de\le \frac 12-\frac 1{p+1}$.

Actually, under the assumption of \eqref{eq-data4-h}, we can solve  \eqref{eq-wave-h},  with sufficiently small $\ep$,
by a contraction mapping argument. For any $u\in L^{p+1}(\R^+\times\H^n)$, we define $w=T[u]$ as the solution of
\beeq\label{eq-5.1}\Box_{\H^n}w
=F_p(u),
\eneq
with initial data $(\ep u_0(x), \ep u_1(x))$.
With help of \eqref{eq-homoStri2} and \eqref{eq-Strichartz},  we know that there exists a constant $C>0$ such that
$$
\|T[u]\|_{L^{p+1}(\R^+\times\H^n)}\le C \ep
+C\|F_p(u)\|_{L^{(p+1)/p}(\R^+\times\H^n)}\le C\ep+C' 
\|u\|_{L^{p+1}(\R^+\times\H^n)}^p,
$$
\begin{eqnarray*}
\|T[u]-T[v]\|_{L^{p+1}(\R^+\times\H^n)} & \le &C
\|F_p(u)-F_p(v)\|_{L^{(p+1)/p}(\R^+\times\H^n)} \\
& \le & 
C'' (\|u\|_{L^{p+1}}+
\|v\|_{L^{p+1}})^{p-1}
\|u-v\|_{L^{p+1}}.
\end{eqnarray*}
Thus $T$ is a contraction mapping on the complete set
$$\{u\in L^{p+1}(\R^+\times\H^n), \, 
\|u\|_{L^{p+1}(\R^+\times\H^n)}\le 2C\ep
\},$$
provided that
$C'(2C\ep)^p\le C\ep$ and
$C''(4C\ep)^{p-1}\le 1/2$,
which are be ensured if we assume
 $$ \ep \le 
(2^pC'
+2C'')^{-1/(p-1)}C^{-1}.$$

\medskip
\noindent {\bf Remark.}  It would be interesting to see if there were an analog of Theorem~\ref{GWPsmooth} for spaces of variable
curvature.  Specifically, if $(M,g)$ is a simply connected and complete Riemannian manifold of dimension $n\ge2$ and
has sectional curvatures satisfying $K\in [-\kappa_1,-\kappa_0]$ for some $\kappa_1>\kappa_0>0$ and if $p>1$, are there
always global solutions to the equation
$$(\partial_t^2-\Delta_g+\rho^2 \kappa_0)u=F_p(u), \quad \rho=(n-1)/2,$$
for sufficiently small initial data with fixed compact support?  Note that Corollary~\ref{thm-NegCurv} says that such
a result is true if $\rho^2\kappa_0$ is replaced by any larger constant $k$.  Also, as we mentioned before, one needs
$p>p_S(n)$ for this to be true for $\kappa_0=0$ due to what happens for the standard d'Alembertian in 
Minkowski space, and thus the assumption that $(M,g)$ be negatively curved is needed.

In practice we can always take $\kappa_0=1$.  In this case, a perhaps harder problem would be whether one
has dispersive estimates as in Lemma~\ref{thm-dispersive} assuming that $K\le -1$ and $\inf K>-\infty$.  
Due to properties of the leading term in the Hadamard parametrix (see e.g., \cite{Hangzhou}),
this problem
seems to be related to classical Riemannian volume comparison theorems and the Cartan-Hadamard conjecture.

\section{Super-conformal case}\label{sec:superconformal}
In this section, we use Strichartz type estimates and Sobolev embedding to prove Theorem \ref{thm-superconformal}.


First, by \eqref{eq-2.7'} with $p= p_{\rm conf}$, we know that
for
$q=2(n+1)/(n-1)$ with $q=q' p_{\rm conf}$,
 \beeq\label{eq-2.7''}
\|S(\tau)f\|_{L^q(\H^n)}\les 
\, K_{q}(\tau)\|f\|_{L^{q'}(\H^n)},
\eneq
where 
$K_q(\tau)$ is given in \eqref{eq-5.2}.

Based on  Duhamel's principle and \eqref{eq-2.7''}, we see that 
$$\|u(\tau)\|_{L^{q}(\H^n)}\le C
\int_{0}^{\tau}
K_{q}(\tau-s)
\|F(s)\|_{L^{q'}(\H^n)}d s$$
for  solutions to
$\Box_{\H^n} u=F$ with vanishing data at $\tau=0$.
Since for any $N>0$, we have
$$K_{q}(\tau)\chi_{\tau>0}\les |\tau|^{-(n-1)(1/2-1/q)}(1+|\tau|)^{-N}
=|\tau|^{-(n-1)/(n+1)}(1+|\tau|)^{-N}
,$$ 
by the Hardy-Littlewood-Sobolev inequality and Young's inequality,
we obtain the inhomogeneous Strichartz estimates 
\beeq\label{eq-Strichartz2}\|u\|_{L^{p_{0}}_{\tau}L_{x}^{q}(\R^{+}\times\H^n)}\le C
\|F\|_{L_{\tau}^{p_{1}'}L_{x}^{q'}(\R^+\times\H^n)},
\eneq
provided
that
\beeq\label{eq-Strichartz3} p_{0}, p_{1}\in [1,\infty], \ \frac{1}{p_{0}}+\frac{1}{p_{1}}\in [\frac{n-1}{n+1}, 1]\ .\eneq
If we combine it with the homogeneous estimates,
Lemma \ref{thm-homoStri}, we get
\beeq\label{eq-Strichartz4}\|u\|_{L^{p_{0}}_{\tau}L_{x}^{q}(\R^{+}\times\H^n)}\les
\|D^{s_0} u(0)\|_{L^{r_0'}(\H^n)}+
\|D^{s_0-1}u_\tau (0)\|_{L^{r_0'}(\H^n)}+
\|\Box_{\H^{n}}u\|_{L_{\tau}^{p_{1}'}L_{x}^{q'}(\R^+\times\H^n)},
\eneq
provided
that
$r_{0}\in (2, q]$, $p_{0}\in [1,\infty)$, $p_{1}\in [1,\infty]$
 and 
\beeq\label{eq-Strichartz5} s_0=\frac{n+1}{2}-\frac{1}{r_0}-\frac{n}{q},\ \frac{1}{p_{0}}+\frac{1}{p_{1}}\in [\frac{n-1}{n+1}, 1],\ 
\frac 1{p_{0}}>(n-1)\left(\frac 12-\frac 1{r_0}\right)
\ .\eneq

To prove 
Theorem \ref{thm-superconformal}
where
 $p\in (p_{\rm conf}, 1+\frac{4n}{n^2-3n-2}]$,
 we 
 choose
 $$p_{0}=\frac{(n+1)(p-1)}2, p_{1}'=\frac{p_{0}}p,
r_{0}=\frac{2}{1-2\de},
s_{0}=\frac{n}{n+1}+\de,
 $$
 with $\de\in (0, \frac 2{(n^{2}-1)(p-1)})$, and we have
\beeq\label{eq-Strichartz6}
\|D u\|_{L^{p_{0}}_{\tau }L_{x}^{q}}\les 
\|D^{s_0+1} u(0)\|_{L^{2/(1+2\de)}}+
\|D^{s_0}u_\tau (0)\|_{L^{2/(1+2\de)}}+
\|D F_{p}(u)\|_{L_{\tau}^{p_{0}/p}L_{x}^{q'}}.
\eneq
Thus we see that to finish the proof
as in Section \ref{sec:altprf}, we need only to prove the following nonlinear inequality
\beeq\|D F_{p}(u)\|_{L_{x}^{q'}}
\les \|D u\|^{p}_{L_{x}^{q}}
\ .\eneq

Actually, by \eqref{eq-Fp}, we have the chain rule
\beeq\|D F_{p}(u)\|_{L_{x}^{q'}}
\les \|D u\|_{L_{x}^{q}}
\| u\|^{p-1}_{L_{x}^{(n+1)(p-1)/2}}
\ .\eneq
Moreover, the Sobolev embedding gives us
$$\| u\|_{L_{x}^{(n+1)(p-1)/2}}\les\|D u\|_{L_{x}^{q}}$$
provided that
$$1\ge \frac{n}{q}-\frac{n}{(n+1)(p-1)/2}$$
which is ensured by our assumption that
 $p\in (p_{\rm conf}, 1+\frac{4n}{n^2-3n-2}]$.

\section{Appendix}\label{appendix}

It was pointed out by the referee that, owing to the fact that Bessel potentials of order one in ${\mathbb R}$ have
a logarithmic singularity at the origin (and infinity), Metcalfe and Taylor in \cite{MT} assert that 
Tataru's~\cite{Tataru} proof of the dispersive
estimates \eqref{eq-2.7} and \eqref{eq-2.8} has a small gap when $n=3$.  
As a result, we felt it necessary to give an independent proof of these results here.  As we shall point out at the
end of this Appendix, our proof involves essentially the same analytic family of operators as was employed by
Tataru.  We choose our normalization to highlight how the unboundedness of this particular Bessel potential does
not lead to the gap in Tataru's~\cite{Tataru} proof that was asserted in \cite{MT}.

For simplicity (and since the assertion about the ``gap'' was for $n=3$)
we shall just treat the three-dimensional case of ${\mathbb H}^3$ for the dispersive estimates
in Lemma~\ref{thm-dispersive}.  After the proof we shall explain that the lack of boundedness of Bessel potentials
of certain critical orders does not cause any problems for $n\ge2$.  Incidentally, Tataru's~\cite{Tataru} analytic
family of operators is set up in a way that is very similar to the ones used by Strichartz~\cite{Strichartz} in
the proof of the original  ``Strichartz estimates''.

As before, $D_0=\sqrt{-\Delta_{{\mathbb H}^3}-\rho^2}$ denotes the square root of minus the shifted Laplacian, with $\rho=(n-1)/2$ being equal to one
in this case since we are working in ${\mathbb H}^n$ with $n=3$.
We then make a slight modification of Tataru's argument by introducing the two analytic family of operators
\begin{equation}\label{a.1}
S_z(t)=(z+1)e^{z^2} D^z \, \frac{\sin tD_0}{D_0},
\end{equation}
and
\begin{equation}\label{a.2}
C_z(t)=(z+1)e^{z^2} D^{-1+z} \, \cos tD_0,
\end{equation}
where, as before, $D=\sqrt{-\Delta_{{\mathbb H}^n}}$.  We have  included the crucial factor $(z+1)$ in the definition of $S_z$ to compensate for the
aforementioned unboundedness of the Bessel potential of order one.

To prove the dispersive estimates for ${\mathbb H}^3$ we then require the following bounds which follow from well known properties of Bessel potentials and
the spectral theorem.

\begin{prp}\label{propa}  There is a uniform constant $C$ so that for $t>0$
\begin{equation}\label{a.3}
\|S_z\|_{L^1({\mathbb H}^3) \to L^\infty({\mathbb H}^3)}, \, \, 
\|C_z\|_{L^1({\mathbb H}^3) \to L^\infty({\mathbb H}^3)} \, \, \le C/\sinh t, \, \, \text{if } \, \, 
\text{Re }z=-1,
\end{equation}
and also
\begin{equation}\label{a.4}
\|S_z(t) \|_{L^2({\mathbb H}^3)\to L^2({\mathbb H}^3)}\le C(1+t), \, \, \,
\|C_z(t) \|_{L^2({\mathbb H}^3)\to L^2({\mathbb H}^3)}\le C, \quad
\text{if } \, \, \text{Re }z=1.
\end{equation}
\end{prp}

By analytic interpolation, \eqref{a.3} and \eqref{a.4} yield \eqref{eq-2.7} and \eqref{eq-2.8}.

Since \eqref{a.4} trivially follows from the spectral theorem, we just need to prove \eqref{a.3}, which
amounts to showing that the kernels of $S_z(t)$ and $C_z(t)$ are $O(1/\sinh t)$ when $\text{Re }  z=-1$.

To prove these bounds we shall make use of the following simple lemma.

\begin{lem}\label{lemmaa}
Let
$$G_z(r)=(z+1)e^{z^2}\int_{-\infty}^\infty (1+\eta^2)^{z/2} \, e^{i\eta r} \, d\eta,$$
and
$$H_z(r)=(z+1)e^{z^2}\int_{-\infty}^\infty \eta\, (1+\eta^2)^{-1/2+z/2} \, e^{i\eta r}\, d\eta.$$
Then if $r\ne 0$
\begin{equation}\label{a.5}
|G_z(r)|, \, \, |H_z(r)|\le Ce^{-|r|}, \quad \text{if } \, \, \text{Re }z=-1.
\end{equation}
\end{lem}

To prove this, we will use the well known formula for Bessel potentials (see e.g., \cite{AS}),
\begin{align}\label{a.6}
&\int_{-\infty}^\infty (1+\eta^2)^{-z/2} e^{i\eta r}\, d \eta
\\
&\, \, =
\frac{2\pi e^{-|r|}}{2^{z/2}\Gamma(z/2)\Gamma(1-z/2)} \, 
\int_0^\infty e^{-|r| \tau}\bigl(\tau+\tau^2/2)^{-z/2}\, d\tau \notag
\\
&\, \, =
\frac{2\pi e^{-|r|}}{2^{z/2}\Gamma(z/2)} \, \times \, \frac1{\Gamma(1-z/2)}
\int_0^\infty \tau^{-z/2} \, \Bigl((1+\tau/2)^{-z/2} e^{-|r|\tau}\Bigr) \, d\tau,
\, \, \, r\ne 0.
\notag
\end{align}

This would  leads to the bound for $G_z$ in \eqref{a.5} if we could show that 
\begin{multline}\label{7.5'}\tag{7.5$'$}\frac{se^{-s^2}}{2^{(1+is)/2}\Gamma((1+is)/2) \, \Gamma((1-is)/2)}
\int_{0}^\infty  e^{-|r|\tau} \bigl(\tau+\tau^2/2\bigr)^{-(1+is)/2} \, d\tau 
\\
=O(1).
\end{multline}

Note that if $r\ne 0$ is fixed, the first and last terms  in \eqref{a.6} are entire functions of $z\in {\mathbb C}$.  There is no problem interpreting
the first term, since it is a standard oscillatory integral, while, because of the $1/\Gamma(1-z/2)$ factor, the last term in 
 \eqref{a.6} is well defined as a  standard Riemann-Liouville integral.
 
In \cite{AS}, \eqref{a.6} is stated for 
$z\in (1,2)$
 (i.e., (2, 10)--(3, 6) in \cite{AS}), which, by analytic continuation, of course implies the formula for
$\text{Re }z\in (1, 2)$.
   If $r\ne 0$ is fixed, then, as we just mentioned, the first term
in this formula and the last term are entire functions of $z$, and so the formula is valid for all $z\in {\mathbb C}$.

To verify \eqref{7.5'}, fix $\beta\in C^\infty(\R)$ satisfying $\beta(\tau)=0$ for $\tau\le1$ and $\beta(\tau)=1$ for
$\tau\ge 2$.  Then clearly the left side of \eqref{7.5'} is of the form
\begin{align*}
&\frac{se^{-s^2}}{2^{(1+is)/2}\Gamma((1+is)/2) \, \Gamma((1-is)/2)}
\int_{0}^\infty \beta(\tau)  e^{-|r|\tau} \bigl(\tau+\tau^2/2\bigr)^{-(1+is)/2} \, d\tau +O(1)
\\
&\qquad=\frac{se^{-s^2}}{\Gamma((1+is)/2) \, \Gamma((1-is)/2)}\int_0^\infty \beta(\tau) e^{-|r|\tau}
\tau^{-1-is} \, d\tau +O(1)
\\
&\qquad=\frac{ie^{-s^2}}{\Gamma((1+is)/2) \, \Gamma((1-is)/2)}\int_0^\infty \beta(\tau)e^{-|r|\tau}
\frac{d}{d\tau} \tau^{-is} \, d\tau+
O(1)
 =O(1),
\end{align*}
as claimed in \eqref{7.5'}.

To prove the bounds for $H_z(r)$ in \eqref{a.5} , we first note that
\begin{equation}\label{H.formula}
H_z(r)=\tfrac{-ie^{-2z-1}}{z+2}rG_{z+1}(r), \, \, r\ne 0, \, \, \end{equation}
due to the fact that if $m\in {\mathcal S}'(\R)$, then the  Fourier transform of
$\tfrac\partial{\partial \eta}m(\eta)$ is $ir\Hat m(r)$, where $\Hat m\in {\mathcal S}'(\R)$ is
the Fourier transform of $m$.
Consequently, we would obtain the bounds for $H_z(r)$ in \eqref{a.5} if
$$e^{-s^2}r\int_{-\infty}^\infty (1+\eta^2)^{is/2}e^{i\eta r}=O(e^{-|r|}).$$
This clearly follows from \eqref{a.6}
and the fact that
$$\Bigl| \, r\int_0^\infty e^{-|r|\tau} \bigl(\tau+\tau^2/2\bigr)^{-is/2} \, d\tau \, \Bigr|
\le \int_0^\infty |r| \, e^{-|r|\tau} \, d\tau =1,
$$
which completes the proof.

We can also give another simple proof of the bounds for $H_z$ in \eqref{a.5} using well-known formulae for Bessel potentials in ${\mathbb R}^3$.  We first observe that
\begin{align*}
H_z(r)&=2
i
(z+1)e^{z^2} \int_0^\infty \rho (1+\rho^2)^{-(1-z)/2} \, \sin (r\rho) \, d\rho
\\
&=(2\pi)^{-1} i(z+1)e^{z^2} 
r
\int_{{\mathbb R}^3} (1+|\xi|^2)^{-(1-z)/2} \, e^{i\langle \xi, (r,0,0)\rangle} \, d\xi,
\end{align*}
due to the fact that, if $d\sigma$ denotes surface measure on the two-sphere, we have
$$\frac{|x|}{4\pi} \int_{S^2} e^{ix\cdot \omega} \, d\sigma(\omega) = \sin |x|.$$
Thus, by standard formulae for Bessel potentials (i.e.,  (2, 10)--(3, 6) in \cite{AS})), we have, as before, 
\begin{multline}\label{3d}H_z(r)=
2\pi  i (z+1)e^{z^2}
\frac{e^{-|r|}}{2^{(1-z)/2} \Gamma((1-z)/2)\Gamma((3+z)/2)} \\
\times
 r
\int_0^\infty e^{-|r|\tau} \bigl(\tau+\tau^2/2\bigr)^{(1+z)/2} \, d\tau\ .
\end{multline}
 Using this formula, we can easily obtain another proof of the 
bounds for $H_z$ in \eqref{a.5}.

\bigskip
\noindent{\bf Remark.}  Note that the $(z+1)$ factor is needed in the definition of $G_z$  to
ensure that \eqref{a.5} holds near $z=-1$.   In particular, at the value of $z=-1$ we have $G_{-1}\equiv 0$.  
Additionally, since $G_0(r)=2\pi\delta_0(r)$ and $r\delta_0(r)\equiv 0$, \eqref{H.formula} yields $H_{-1}(r)\equiv0$, which (by what follows) is necessary as $C_{-1}(t)\equiv 0$ in \eqref{a.2}.  On the other hand, even though we included the $(z+1)$ factor in the definition of $H_z$ for consistency, \eqref{3d} shows that it is superfluous.

\bigskip

\begin{proof}[End of proof of Proposition~\ref{propa}]
If $m(\tau)$ is an even function of $\tau$, the operator $m(D_0)$ has kernel
$$\frac{-1}{4\pi^2 \sinh r}\, \frac{\partial}{\partial \tau} \Hat m(\tau)|_{\tau=r}, \quad
r=d_g(x,y),$$
where $d_g(x,y)$ is the hyperbolic distance and $\Hat m$ is the Fourier transform of $m$.  (See \cite{Taylor}.)

Thus, the kernel of $S_z(t)$ equals
$$\frac{-(z+1)e^{z^2}}{4\pi^2i\sinh r} \int_{-\infty}^\infty (1+\eta^2)^{z/2}\sin t\eta \, e^{-ir \eta} \, d\eta.$$

By Lemma~\ref{lemmaa} this is $O((\sinh t)^{-1})$ as desired 
when $\text{Re }z=-1$ if either $r\ge 1$ or $r\ge t/2$.  For the remaining case where $0<r<1$ and $2r<t$ 
we claim that
\begin{equation}\label{a.7}
\Bigl| \, \frac{se^{-s^2}}r \int_{-\infty}^\infty (1+\eta^2)^{-1/2+is/2} \, \sin t\eta \, e^{-i\eta r}\, d\eta\, \Bigr|
\le C(\sinh t)^{-1}, \quad s\in {\mathbb R},
\end{equation}
which would finish the proof of the bounds for $S_z$ in \eqref{a.3}.

Using \eqref{a.6}, Euler's formula and the mean value theorem shows that this bound is valid
in this case when $t\ge 2$.  So we would be done with our $L^1({\mathbb H}^3)\to L^\infty({\mathbb H}^3)$
bounds for $S_z(t)$, $\text{Re }z=-1$, if we could show that the left side of \eqref{a.7} is
$O(1/t)$ when $0<2r<t<2$.

To prove this fix an even nonnegative function $\rho\in C^\infty({\mathbb R})$ which vanishes for $\eta\in (-1,1)$
and equals one when $|\eta|\ge2$.  Then since $\eta\to (1+\eta^2)^{-1/2+is/2}\sin t\eta$ is odd the left side
of \eqref{a.7} equals
\begin{multline*}\Bigl|\, \frac{se^{-s^2}}r\int_{-\infty}^\infty (1+\eta^2)^{-1/2+is/2}\sin t\eta \, \sin r\eta \, d\eta\, \Bigr|
\\
= \, 
\Bigl|\, \frac{se^{-s^2}}r \int_{-\infty}^\infty \rho(t\eta)(1+\eta^2)^{-1/2+is/2} \, \sin r\eta \, \sin t\eta \, d\eta\, \Bigr|
+O(1/t).\end{multline*}

If we integrate by parts, we find that the first term in the right is majorized by
\begin{multline*}
\Bigr|\, \frac{se^{-s^2}}t\int_{-\infty}^\infty \rho(t\eta) \, (1+\eta^2)^{-1/2+is/2} \, \cos r \eta \, \cos t\eta  \, d\eta\, \Bigr|
\\
+ \, \Bigl|\, \frac{se^{-s^2}}{tr} \int_{-\infty}^\infty \frac{d}{d\eta}\bigl(\, \rho(t\eta)(1+\eta^2)^{-1/2+is/2}\, \bigr)
\, \sin r\eta \, \cos t\eta \, d\eta \, \Bigr|.
\end{multline*}
By Euler's formula and a simple integration by parts argument, since $0<2r<t<2$, the first term is $O(1/t)$ as
desired.  If we integrate by parts one more time, we find that the remaining term is
$$\Bigl|\, 
\frac{se^{-s^2}}{t^2r}\int_{-\infty}^\infty \frac{d^2}{d\eta^2}\bigl(\, \rho(t\eta)(1+\eta^2)^{-1/2+is/2}\, \bigr)
\, \sin r\eta \sin t\eta \, d\eta \, \Bigr| +O(1/t).$$
Since the first term here is also clearly $O(1/t)$, this finishes the proof of our bounds for $S_z(t)$
in \eqref{a.3}.

Since the proof of the bounds for $C_z(t)$ in \eqref{a.3} follows from the same argument,
the proof of Proposition~\ref{propa} is complete.
\end{proof}

\medskip
\noindent {\bf Remarks.}  To prove the dispersive estimates in Lemma~\ref{thm-dispersive} when
$n=3$, Tataru~\cite{Tataru} used the analytic family of operators
\begin{equation}\label{a.8}
a(z)\bigl(D_0^2+\beta^2\bigr)^{z/2} \, \frac{\sin tD_0}{D_0},
\end{equation}
for fixed $\beta>\rho$ (which is more favorable than the case $\beta=\rho$ treated above), with, {\em crucially},
\begin{equation}\label{a.9}
a(z)=\frac{e^{z^2}}{\Gamma(z+\rho)}, \quad \rho=(n-1)/2.
\end{equation}
In \cite{MT} it was noted that when $n=3$ the operators
\begin{equation}\label{a.10}
\bigl(D_0^2+\beta^2\bigr)^{z/2} \, \frac{\sin tD_0}{D_0}, \quad t\ne 0,
\end{equation}
{\em do not} map $L^1({\mathbb H}^3)\to L^\infty({\mathbb H}^3)$ if $z=-1$.  This is due to the fact
that for any fixed $\beta>0$ the Bessel potential
\begin{equation}\label{a.11}
r\to \int_{-\infty}^\infty \bigl(\beta^2+\eta^2\bigr)^{z/2} \, e^{i\eta r}\, d\eta
\end{equation}
is not in $ L^\infty({\mathbb R})$, if  $z=-1$.
Indeed, for instance, this particular Bessel potential is $\approx |\ln r|$ for $r>0$ near the origin.

We have to point out that inequality (27) in Tataru~\cite{Tataru}, which is disputed on p.~3496 of
\cite{MT}, does {\em not} lead to their (3.26), which is the assertion that the operators in \eqref{a.10}
do map $L^1({\mathbb H}^3)\to L^\infty({\mathbb H}^3)$ with norm $O(1/\sinh |t|)$.  It is this {\em incorrect
inequality} (i.e., (3.26) in \cite{MT}) that lead the authors in \cite{MT} to say that there is a gap in
Tataru's argument.

To be more specific, Tataru {\em never} claims that this inequality is valid, and, {\em moreover}, his proof
does {\em not} imply this fallacious inequality.  Indeed, what Tataru proves is that the operators defined
in \eqref{a.8}--\eqref{a.9} satisfy
\begin{equation}\label{a.12}
\bigl\| \, a(z) \, \bigl(D^2_0+\beta^2\bigr)^{z/2} \, \frac{\sin tD_0}{D_0} \, \bigr\|_{L^1({\mathbb H}^3)
\to L^\infty({\mathbb H}^3)} \le C/\sinh |t|, \quad \text{if } \, \, \text{Re }z=-1.
\end{equation}
It seems the authors in \cite{MT} overlooked the fact that $z\to 1/\Gamma(z+1)$  behaves
like $(z+1)$ near $z=-1$, and, consequently
$$a(-1)=0.$$
Thus, since $0\cdot \ln r\equiv 0$, $r>0$, Tataru's assertion \eqref{a.12} when $z=-1$ is {\em not}
(3.26) in \cite{MT} since
$$a(z)\bigl(D_0^2+\beta^2\bigr)^{z/2} \, \frac{\sin tD_0}{D_0} \equiv 0, \quad
\text{if } \, z=-1, \, \, \, \text{and } \, t\ne 0,$$
which means that Tataru's estimate \eqref{a.12} is trivial in this disputed case.

We proved \eqref{a.12} above when $\beta=\rho=1$ and $n=3$, with, in our case
$$a(z)=(z+1)e^{z^2}.$$
Just as Tataru's proof relies crucially on the holomorphic damping factor $1/\Gamma(z+1)$, ours used,
in a critical way, the damping factor $(z+1)$ so that, in our case, $a(z)=0$ for $z=-1$.

In any dimension $n\ge 2$, Tataru's normalizing factor $a(z)$ vanishes when $z=-\rho$ and thus, 
like the related original argument of Strichartz~\cite{Strichartz}, avoids possible problems that could
arise from Bessel potentials of order one in ${\mathbb R}$ being unbounded.

As we said before, we are grateful to the referee for bringing to our attention the potential (but non-existent)
problems of our use of Lemma~\ref{thm-dispersive} so that we could address this issue in hopes
that the minor oversight in \cite{MT} about the (non-existent) ``gap'' in Tataru's proof of this estimate is not
propagated further.



\end{document}